\documentclass[11pt,a4paper]{amsart}

\usepackage[latin1]{inputenc}
\usepackage{amsmath,amsfonts,amssymb}
\usepackage{enumerate}
\usepackage[
      colorlinks=true,    %no frame around URL
      urlcolor=blue,    %no colors
      menucolor=blue,    %no colors
      linkcolor=blue,    %no colors
      bookmarks=true,    %tree-like TOC
      bookmarksopen=true,    %expanded when starting
      hyperfootnotes=false,    %no referencing of footnotes, does not compile
      pdfpagemode=UseOutlines    %show the bookmarks when starting the pdf viewer
]{hyperref}
\usepackage[margin=2.5cm]{geometry}
\newtheorem{theorem}{Theorem}[section]
\newtheorem{lemma}[theorem]{Lemma}
\newtheorem{proposition}[theorem]{Proposition}

\newtheorem{corollary}[theorem]{Corollary}

\theoremstyle{definition}
\newtheorem{definition}{Definition}

\numberwithin{equation}{section}
\numberwithin{theorem}{section}
\numberwithin{definition}{section}

\newcommand{\R}{\mathbb{R}}
\newcommand{\Z}{\mathbb{Z}}
\newcommand{\Q}{\mathbb{Q}}
\newcommand{\C}{\mathbb{C}}
\newcommand{\A}{\mathfrak{A}}
\newcommand{\B}{\mathfrak{B}}
\newcommand{\fN}{\mathfrak{N}}
\newcommand{\fp}{\mathfrak{p}}
\newcommand{\cN}{\mathcal{N}}
\newcommand{\cH}{\mathcal{H}}
\newcommand{\cL}{\mathcal{L}}
\newcommand{\z}{\boldsymbol{z}}
\newcommand{\x}{\boldsymbol{x}}

\newcommand{\bo}[1]{\boldsymbol{#1}}
\newcommand{\mc}[1]{\mathcal{#1}}

\def\Vol{\textup{Vol}}
\def\lg{\left\lbrace}
\def\rg{\right\rbrace}

\newcommand{\tN}{\widetilde{N}}

\newcommand{\Oseen}{\mathcal{O}}
\newcommand{\ks}{k}
\newcommand{\e}{e}
\newcommand{\n}{n}
\newcommand{\en}{n}
\newcommand{\m}{m}
\newcommand{\el}{l}
\newcommand{\eL}{L}

\newcommand{\Qbar}{\overline{\Q}}

\newcommand{\ord}{\text{ord}}
\newcommand{\fin}{\text{fin}}

 \author{Fabrizio Barroero}

\address{Scuola Normale Superiore\\
Piazza dei Cavalieri 7, 56126 Pisa\\
Italy}
 \email{fbarroero@gmail.com}

\title{Algebraic $S$-integers of fixed degree and bounded height}

\date{\today}

\thanks{F. Barroero is supported by the Austrian Science Foundation (FWF) project W1230-N13 and ERC-Grant No. 267273.}


\begin{document}

\begin{abstract}
Let $\ks$ be a number field and $S$ a finite set of places of $\ks$ containing the archimedean ones. We count the number of algebraic points of bounded height whose coordinates lie in the ring of $S$-integers of $\ks$. Moreover, we give an asymptotic formula for the number of $\overline{S}$-integers of bounded height and fixed degree over $\ks$, where $\overline{S}$ is the set of places of $\overline{\ks}$ lying above the ones in $S$.
\end{abstract}

\maketitle

\section{Introduction} \label{intro}

In this article we give asymptotic estimates for the cardinality  of certain subsets of $\overline{\Q}^\n$ of bounded height. Here and in the rest of the article, by height we mean the multiplicative absolute Weil height $H$ on the affine space $\Qbar^\n$, whose definition will be recalled in Section \ref{prelim}.

Let $\ks$ be a number field of degree $\m$ over $\Q$ and let $\n$ and $\e$ be positive integers. We fix an algebraic closure $\overline{\ks}$ of $\ks$ and set
$$
\ks(\n,\e)=\lg \bo{\alpha} \in \overline{\ks}^\n: [\ks(\bo{\alpha}):\ks]=\e \rg,
$$
where $\ks(\bo{\alpha})$ is the field obtained by adjoining all the coordinates of $\bo{\alpha}$ to $\ks$. By Northcott's Theorem \cite{Northcott1949}, subsets of $\ks(\n,\e)$ of uniformly bounded height are finite. Therefore, for any subset $A$ of $\ks(\n,\e)$ and $\cH> 0$, we may introduce the following counting function 
$$
N(A,\cH)= |\lg \bo{\alpha}\in A: H(\bo{\alpha})\leq \cH  \rg|.
$$
Various results about this counting function appear in the literature. One of the earliest is a result of Schanuel \cite{Schanuel1979}, who gave an asymptotic formula for 
$N(\ks(\n,1),\cH)$. Schmidt was the first to consider the case $\e>1$. In \cite{Schmidt1993}, he found upper and lower bounds for $N(\ks(\n,\e),\cH)$ while in \cite{Schmidt1995}, he gave asymptotics for $N(\Q(\n,2),\cH)$. Shortly afterwards, Gao \cite{Gao1995} found the asymptotics
for  $N(\Q(\n,\e),\cH)$, provided $\n>\e$. Later Masser and Vaaler \cite{Masser2007} established an asymptotic estimate for $N(\ks(1,\e),\cH)$. 
Finally, Widmer \cite{Widmer2009} proved an asymptotic formula for $N(\ks(\n,\e),\cH)$, provided $\n>5\e/2+5+2/\m\e$. 
However, for general $\n$ and $\e$ even the correct order of magnitude for $N(\ks(\n,\e),\cH)$ remains unknown.

In this article we are interested in counting algebraic $S$-integers. Let $S$ be a finite set of places of $\ks$ containing the archimedean ones. As usual $\Oseen_S$ indicates the ring of $S$-integers of $\ks$. Let $\overline{S}$ be the set of places of $ \overline{\ks}$ that lie above the places in $S$ and let $\Oseen_{\overline{S}}$ be the ring of $\overline{S}$-integers of $\overline{\ks}$. Alternatively, we could think of $\Oseen_{\overline{S}}$ as the ring of those algebraic numbers having minimal polynomial over $k$ that is monic and has coefficients in $\Oseen_S$.

Given $n$ and $e$ positive integers, we set 
$$
\Oseen_S(\n,\e)= \ks(\n,\e) \cap \Oseen_{\overline{S}}^\n = \lg \bo{\alpha} \in \Oseen_{\overline{S}}^\n:  [\ks(\bo{\alpha}):\ks]=\e  \rg.
$$
Let $S_\infty$ be the set of archimedean places of $\ks$. If we choose $S=S_\infty$, then $\Oseen_S=\Oseen_\ks$ is the ring of algebraic integers of $\ks$ and we use the notation $\Oseen_\ks(\n,\e)$ with the obvious meaning.
Besides the trivial cases $\Oseen_\Q(\n,1)=\Z^\n$, the first asymptotic result can probably be found in Lang's book  \cite{Lang1983}. Lang states, without proof,  
\begin{equation*}
N(\Oseen_\ks(1,1),\cH)=\gamma_\ks \cH^\m  \left( \log \cH  \right)^q + O\left(    {\cH}^{\m} \left( \log  \mc{H}  \right)^{q-1} \right),
\end{equation*}
where $m=[\ks:\Q]$, $q$ is the rank of the unit group of $\Oseen_\ks$, and $\gamma_\ks$ and the implicit constant in the error term are unspecified positive constants, depending on $\ks$. More recently,  
Widmer \cite{Widmer2013} established the following asymptotic formula 
\begin{alignat}1\label{Okne}
N(\Oseen_\ks(\n,\e),\cH)=\sum_{i=0}^{t} D_i\cH^{\m\e\n}(\log \cH^{\m\e\n})^i+O_{m,e,n}(\cH^{\m\e\n-1}(\log \cH)^{t}),
\end{alignat}
provided $\e=1$ or $\n>\e+C_{\e,\m}$, for some explicit $C_{\e,\m}\leq 7$.
Here $t=\e(q+1)-1$, and the constants $D_i=D_i(\ks,\n,\e)$ are explicitly given. Our Theorem \ref{thmvect} generalizes Widmer's result in the case $\e=1$ to asymptotics for $N(\Oseen_S(\n,1),\cH)$. However, we do not obtain a multiterm expansion as in (\ref{Okne}).

Chern and Vaaler, in \cite{Chern2001}, proved an asymptotic formula for the 
number of monic polynomials in $\Z[X]$ of given degree and bounded Mahler measure. Theorem 6 of \cite{Chern2001} immediately implies the following estimate 
\begin{equation*}\label{chern}
N(\Oseen_\Q(1,\e),\cH)= C_{\e} \cH^{\e^2}+  O_e \left(    \mc{H}^{\e^2-1} \right),
\end{equation*}
for some explicit constant $C_{e}$. This was extended by the author in \cite{algint}, where an asymptotic estimate is given for $N(\Oseen_\ks(1,\e),\cH)$. Our Theorem \ref{thmint} generalizes this result and gives an asymptotic estimate for $N(\Oseen_S(1,\e),\cH)$ for any finite set of places $S$ containing the archimedean ones.

%As to the author knowledge, this work is the first attempt to find asymptotic formulas for $N(\Oseen_S(\n,\e),\cH)$, when $S$ does not contains archimedean places only. 

We write $S_\fin$ for the set of non-archimedean places of $S$. Suppose that $S_\fin=\{v_1 ,\dots ,v_\eL \}$ and that $v_\el$ corresponds to the prime ideal $\fp_\el$ of $\Oseen_\ks$. We indicate by $\fN(\A)$ the norm from $\ks$ to $\Q$ of the fractional ideal $\A$ and by $\fN(S)$ the $L$-tuple $(\fN(\fp_1) , \dots , \fN(\fp_L))$. Let $r$ and $s$ be, respectively, the number of real and pairs of conjugate complex embeddings of $\ks$. Moreover, we indicate by $\Delta_\ks$ the discriminant of $\ks$.
Let $\n$ be a positive integer, we set
\begin{equation}\label{const}
 B^{(\n)}_{\ks,S}= \frac{\n^{r+s-1}  2^{s\n} \m^{|S|-1}}{(|S|-1)!\left( \sqrt{|\Delta_\ks|}\right)^{\n}}  \prod_{\el=1}^\eL \left( \frac{1}{ \log \fN(\fp_l)} \left( 1-\frac{1}{\fN(\fp_l)^\n} \right)\right),
\end{equation}
and
$$
C_{\R,\n} = 2^{\n-M} \left( \prod_{j=1}^{M} \left( \frac{2j}{2j+1} \right)^{\n-2j} \right) \frac{\n^M}{M!},
$$
with $M=\lfloor \frac{\n-1}{2} \rfloor$ (as usual $\lfloor x \rfloor$ indicates the floor of $x\in \R$), and 
$$
C_{\C,\n} = \pi^\n \frac{\n^\n}{\left( \n! \right)^2}.
$$
In this article, as usual, empty products are understood to be 1. 

For non-negative real functions $f(X),g(X),h(X)$ and $X_0\in \R$,  we write 
$f(X)=g(X)+O(h(X))$ as $X\geq X_0$ tends to infinity, if there is $C_0$ such that $|f(X)-g(X)|\leq C_0 h(X)$ for all $X \geq X_0$.

\begin{theorem}\label{thmvect}
Let $\n$ be a positive integer and let $\ks$ be a number field of degree $\m$ over $\Q$. Moreover, let $S$ be a finite set of places of $\ks$ containing the archimedean ones. Then, as $\cH\geq 2$ tends to infinity,  
$$
N(\Oseen_S(\n,1),\cH) =  (2^r\pi^s)^\n B^{(\n)}_{\ks,S}\mc{H}^{\m\n} \left( \log  \mc{H} \right)^{|S|-1}  \\
+\left\lbrace
\begin{array}{ll}
O\left(    \mc{H}^{\m\n} \left( \log  \mc{H}  \right)^{|S|-2} \right),  & \mbox{if $|S|> 1$,}\\
O\left(\mc{H}^{\m\n-1}  \right), & \mbox{if $|S|= 1$.}
\end{array} \right.
$$
The implicit constant in the error term depends on $m$, $n$ and $\fN(S)$.
\end{theorem}

%We set
%$$
%C^{(\e)}_{\ks,S} = \e^{|S|} C_{\R,\e}^r C_{\C,\e}^s B^{(\e)}_{\ks,S}.
%$$

\begin{theorem}\label{thmint}
Let $\e$ be a positive integer and let $\ks$ be a number field of degree $\m$ over $\Q$. Moreover, let $S$ be a finite set of places of $\ks$ containing the archimedean ones. Then, as $\cH\geq 2$ tends to infinity, 
$$
N(\Oseen_S(1,\e),\cH) =  \e^{|S|} C_{\R,\e}^r C_{\C,\e}^s B^{(\e)}_{\ks,S} \mc{H}^{\m\e^2} \left( \log  \mc{H} \right)^{|S|-1} \\
+ \left\lbrace
\begin{array}{ll}
O\left(    \mc{H}^{\m\e^2} \left( \log  \mc{H}  \right)^{|S|-2} \right),  & \mbox{ if $|S|>1$, }\\
O\left(\mc{H}^{\e(\m\e-1)}  \mc{L} \right), & \mbox{ if $|S|=1$, }
\end{array} \right.
$$
where $\mc{L}=\log \cH$ if $(m,e)=(1,2)$ and 1 otherwise. The implicit constant in the error term depends on $m$, $e$ and $\fN(S)$.
\end{theorem}

As mentioned before, if $S=S_\infty$, then Theorem \ref{thmvect} reduces to (\ref{Okne}), although with a larger error term, and Theorem \ref{thmint} to the result in \cite{algint}. However, for the case $S_\infty \neq S$ the results appear to be new.

As in \cite{algint}, our proof relies on a work of the author and Widmer \cite{Barroero2012} about counting lattice points in definable sets in o-minimal structures. Our approach is similar to the one in \cite{algint}, but in the case $S=S_\infty$ the result is more straightforward, because the embedding of $\Oseen_\ks$ in $\R^m$ is a lattice. On the other hand, if $S\supsetneq S_\infty$, the embedding of $\Oseen_S$ is dense in $\R^m$, and a more elaborate proof is needed. 

Let us apply our theorems in a few simple examples. Fix a prime number $p$. One can see, as an easy exercise and as a special case of both theorems, that the number of elements of $\Z\left[\frac{1}{p}\right]$ of height at most $\cH$ is
$$
\frac{2}{\log p} \left(1-\frac{1}{p} \right)\cH \log \cH +O_p(\cH).
$$
Now, let $d$ be a square-free positive integer with $d \equiv 3 \mod 4$. Consider $\ks=\Q(\sqrt{d})$ and set $S$ to consist of the place corresponding to the prime ideal $(2,1+\sqrt{d})$, in addition to the two archimedean places. Then 
$$
N(\Oseen_S(\n, 1),\cH)=\frac{2n(2^n-1)}{ d^{\frac{n}{2}} \log 2 } \cH^{2\n}\left( \log \cH \right)^2 + O_n\left( \cH^{2n} \log \cH \right).
$$
Now consider $\ks=\Q$ again and suppose the non-archimedean places in $S$ are associated to the primes 2 and 3. Then
$$
N(\Oseen_S(1, 2),\cH)=\frac{32}{3 \log 2 \log 3}\cH^4 \left( \log \cH \right)^2 + O\left( \cH^{4} \log \cH \right).
$$

In \cite{Masser2007}, Masser and Vaaler observed that the limit for $
\cH\rightarrow \infty$ of
$$
\frac{N\left(\ks(1,\e),\cH^{\frac{1}{\e}}\right)}{N(\ks(\e,1),\cH)}
$$ 
is a rational number. Moreover, they asked if this can be extended to some sort of reciprocity law, i.e., whether
$$
\lim_{\cH \rightarrow \infty}\frac{N\left(\ks(\n,\e),\cH^{\frac{1}{\e}}\right)}{N\left(\ks(\e,\n),\cH^{\frac{1}{\n}}\right)}\in \Q.
$$
Analogously we notice that
$$
\lim_{\cH \rightarrow \infty} \frac{N\left(\Oseen_S(1,\e),\cH^{\frac{1}{\e}}\right)}{N(\Oseen_S(\e,1),\cH)}=\e\left(\frac{C_{\R,\e}}{2^\e}\right)^r \left(\frac{C_{\C,\e}}{\pi^\e}\right)^s
$$
is a rational number depending only on $e$, $r$ and $s$, as already pointed out in \cite{algint} for the case $S=S_\infty$. As Masser and Vaaler did, one can ask again whether 
$$
\lim_{\cH \rightarrow \infty}\frac{N\left(\Oseen_S(\n,\e),\cH^{\frac{1}{\e}}\right)}{N\left(\Oseen_S(\e,\n),\cH^{\frac{1}{\n}}\right)}\in \Q.
$$

\section{Preliminaries} \label{prelim}

Let $\ks$ be a number field of degree $\m$ over $\Q$ and let $M_\ks$ be the set of places of $\ks$. For $v\in M_\ks$, we indicate by $\ks_v$ the completion of $\ks$ with respect to $v$. We write $\Q_v$ for the completion of $\Q$ with respect to the unique place of $\Q$ that lies below $v$. Moreover, we set $d_v=[\ks_v:\Q_v]$ to be the local degree of $\ks$ at $v$. 

Any $v \in M_\ks$ corresponds either to a non-zero prime ideal $\fp_v$ of $\Oseen_\ks$ or to an embedding of $\ks$ into $\C$. In the first case $v$ is called a finite or non-archimedean place and we write $v\nmid \infty$. In the second case $v$ is called an infinite or archimedean place and we write $v\mid \infty$. We set, for $v\nmid \infty$,
$$
|\alpha|_v=\fN(\fp_v)^{-\frac{\ord_{\fp_v}(\alpha)}{d_v}},
$$
for every $\alpha \in \ks\setminus \{0\}$, where $\ord_{\fp_v}(\alpha)$ is the power of $\fp_v$ in the factorization of the principal fractional ideal $\alpha \Oseen_\ks$. Furthermore, $|0|_v=0$.
If $v \mid \infty$ corresponds to $\sigma_v :\ks \hookrightarrow \C$, we set
$$
|\alpha|_v = |\sigma_v(\alpha)|,
$$
for every $\alpha \in \ks$, where $|\cdot|$ is the usual absolute value on $\C$. The absolute multiplicative Weil height $H:\ks^n \rightarrow [1,\infty)$ is defined by
\begin{equation}\label{defH}
H(\alpha_1, \dots ,\alpha_n)=\prod_{v \in M_\ks} \max \{ 1, |\alpha_1|_v, \dots ,|\alpha_n|_v \}^{\frac{d_v}{\m}}.
\end{equation}
Note that for $\alpha \in \ks\setminus \{0\}$, $|\alpha|_v \neq 1$ for finitely many $v$. Therefore, the above product contains only finitely many terms different from 1.
Moreover, this definition is independent of the field containing the coordinates, and therefore the height is defined on $\overline{\Q}^n$. For properties of the Weil height we refer to the first chapter of \cite{BombGub}.

We conclude this section introducing semialgebraic sets and stating the Tarski-Seidenberg principle.

\begin{definition}
Let $N$ and $M_i$, for $i=1,\dots ,N$, be positive integers. A \emph{semialgebraic subset} of $\R^n$ is a set of the form
$$
\bigcup_{i=1}^{N}\bigcap_{j=1}^{M_i} \{  \x \in \R^n : f_{i,j}(\x) \ast_{i,j} 0 \},
$$
where $f_{i,j} \in \R[X_1, \dots , X_n]$ and the $\ast_{i,j}$ are either $<$ or $=$.

Let $A\subseteq \R^n$ be a semialgebraic set, a function $f:A\rightarrow \R^{n'}$ is called semialgebraic if its graph $\Gamma(f)$ is a semialgebraic set of $\R^{n+n'}$.
\end{definition}

If we identify $\C$ with $\R^2$, then the definitions of semialgebraic set and function are extended to subsets of $\C^n$ and to functions of complex variables in a natural way.
We will need the following theorem which is usually known as the Tarski-Seidenberg principle.

\begin{theorem}[\cite{Bierstone88}, Theorem 1.5]\label{tarski}
Let $A \in \R^{n+1}$ be a semialgebraic set, then $\pi(A)\in \R^n$ is semialgebraic, where $\pi :\R^{n+1}\rightarrow \R^n$ is the projection map on the first $n$ coordinates. 
\end{theorem} 

\section{A generalization}

In this section we formulate a theorem which will be used later to derive Theorems \ref{thmvect} and \ref{thmint}.

In the following definition we consider functions whose domain is $\R^{n+1}$ or $\C^{n+1}$. We use the notation $\z$ to indicate a vector with entries in a generic field, while $\x$ will be a vector with real coordinates. We are often going to identify a function $f:\C^n\rightarrow \R$ with $f:\R^{2n}\rightarrow \R$, where, if $\x=(x_1, \dots, x_{2n})\in \R^{2n}$, $f(\x)=f(x_1+ix_2,\ldots , x_{2n-1}+ix_{2n})$.

\begin{definition}\label{defdist}
Let $n $ be a positive integer. A \emph{semialgebraic distance function} (of dimension $n$) is a continuous function $N$ from $\R^{n+1}$ or $\C^{n+1}$ to the interval $[0,\infty)$ satisfying the following conditions:
\begin{enumerate}[i.]
\item \label{cond1} $N(\z)=0$ if and only if $\z$ is the zero vector;
\item \label{cond2} $N(w \z)=|w|N(\z)$ for any scalar $w$ in $\R$ or in $\C$;
\item \label{semialghyp} $N$ is a semialgebraic function.
\end{enumerate}

\end{definition}

Let $r$ and $s$ be non-negative integers, not both zero. A system $\mathcal{N}$ of $r$ real and $s$ complex semialgebraic distance functions (of dimension $n$) is called $(r,s)$-system (of dimension $n$).

Let us fix a number field $\ks$ with $[\ks:\Q]=m$. Let $r$ and $s$ be, respectively, the number of real and pairs of conjugate complex embeddings of $\ks$. These induce $r+s$ archimedean places of $\ks$, with respective completions $\R$ or $\C$. Given an $(r,s)$-system $\mathcal{N}$ of dimension $n$, we can associate to every archimedean place $v$ a semialgebraic distance function $N_v$ on $\ks_v^{n+1}$. We will mostly use the alternative notation $N_1,\ldots ,N_r$ for the $r$ real distance functions and $N_{r+1},\ldots ,N_{r+s}$ for the $s$ complex ones and we put $d_i=1$, for $i=1,\dots ,r$, and $d_i=2$ for $i=r+1,\dots ,r+s$.
For the non-archimedean places we set
$$
N_v(\z)=\max \left\lbrace |z_0|_v, \ldots , |z_n|_v  \right\rbrace ,
$$
for $\z = (z_0, \dots, z_n) \in \ks_v^{n+1}$.
Now we can define, for $\bo{\alpha} \in \ks^{n+1}$, a height function associated to $\mathcal{N}$,
\begin{equation*}\label{defheight}
H_\mathcal{N} (\bo{\alpha})^\m=\prod_{v \in M_\ks } N_v(\sigma_v(\bo{\alpha}))^{d_v},
\end{equation*}
where $\sigma_v$ is the embedding of $\ks$ into $\ks_v$ corresponding to $v$, extended componentwise to $\ks^{n+1}$.

Now,  let $\Oseen_S^{\mc{N}}(\cH)$ be the set of $\bo{a} \in \Oseen_S^n$ with $H_\mathcal{N} (1,\bo{a}) \leq \mc{H}$. We are interested in obtaining an estimate for $|\Oseen_S^{\mc{N}}(\cH)|$ as $\mc{H}\rightarrow \infty$.

Let us introduce some notation and impose some conditions on the functions $N_i$ in view of the application of this estimate.
For $i=1, \dots ,r+s$, we set $\tN_i(\z)=N_i(1,\z)$ and suppose that
\begin{equation}\label{hyp1}
\tN_i(\z)\geq 1,
\end{equation}
for every $\z \in \R^n$ or $\C^n$. We define the sets
\begin{equation}\label{sets}
Z_i(T)=\lg \z:  \tN_i(\z)\leq T \rg,
\end{equation}
and suppose that 
\begin{equation}\label{hyp2}
\text{the $Z_i(T)$ have volume $p_i(T)$ for every $T\geq 1$} 
\end{equation}
where $p_i(X)\in \R[X]$ is a polynomial of degree $d_i n$ and leading coefficient $C_i$. Moreover, let 
\begin{equation}\label{constmaint}
C_{\cN,\ks,S}=
\frac{n^{r+s-1}  2^{sn} \m^{|S|-1}}{(|S|-1)!\left( \sqrt{|\Delta_\ks|}\right)^{n}} \left( \prod_{i=1}^{r+s} C_i \right) \prod_{\el=1}^\eL \left(\frac{1}{\log \fN(\fp_l)} \left( 1-\frac{1}{\fN(\fp_l)^\n} \right)\right).
\end{equation}

\begin{theorem}\label{mainthm}
 Let $\mc{N}$ be an $(r,s)$-system of dimension $n$, satisfying the above hypothesis \eqref{hyp1} and \eqref{hyp2}. Moreover, suppose $S$ is a finite set of places of $\ks$ containing the archimedean ones. Then, for every $\cH_0>1$ there exists a positive $C_0=C_0(\mc{N},\fN(S),\cH_0)$, such that for every $\cH\geq \cH_0$
\begin{equation*}\label{asymptform}
\left||\Oseen_S^{\mc{N}}(\cH)|- C_{\cN,\ks,S}\mc{H}^{\m n} \left( \log  \mc{H} \right)^{|S|-1} \right| \leq \left\lbrace
\begin{array}{ll}
C_0   \mc{H}^{m n} \left( \log  \mc{H}  \right)^{|S|-2},  & \mbox{if $|S|>1$,}\\
C_0 \mc{H}^{\m n-1}, & \mbox{if $|S|=1$.}
\end{array} \right.
\end{equation*}
\end{theorem}

\section{Proof of Theorems \ref{thmvect} and \ref{thmint}}

In this section we apply Theorem \ref{mainthm} to prove Theorems \ref{thmvect} and \ref{thmint}. Let us start with the first one. We choose our system $\cN$ to consist of the max norm
$$
N_v(\z)=|\z|_\infty=\max \left\lbrace |z_0|, \ldots , |z_n| \right\rbrace ,
$$
for every archimedean place $v$ of $\ks$. These $N_v$ clearly satisfy the definition of semialgebraic distance function. The sets $Z_i(T)$ defined in (\ref{sets}) have volume $(2T)^n$ for $i=1,\dots ,r$ and $\pi^n T^{2n}$ for $i=r+1,\dots ,r+s$, for every $T \geq 1$. Therefore, the hypotheses of Theorem \ref{mainthm} are satisfied.

Note that, for every $\bo{a} \in \ks^n$,
$$
H_\mathcal{N} (1,\bo{a})=\prod_{v } N_v(1,\sigma_v(\bo{a}))^{\frac{d_v}{\m}}=\prod_{v } \max \left\lbrace 1, |a_1|_v, \ldots , |a_n|_v \right\rbrace^{\frac{d_v}{\m}}=H(\bo{a}).
$$
Therefore $H_\cN$ is the usual absolute Weil height defined in (\ref{defH}).
The claim of Theorem \ref{thmvect} follows applying Theorem \ref{mainthm} with $\mc{H}_0=2$.\\

Now let us prove Theorem \ref{thmint}. We choose  $\cN$ to consist of the Mahler measure function:
\begin{equation*}\label{defNn}
N_i(z_0, \dots ,z_n)= M(z_0X^n+z_1X^{n-1}+\cdots +z_n)=M(z_0,\dots, z_n),
\end{equation*}
for every $i=1,\dots ,r+s$.
Let us recall its definition. If $f=z_0X^d+z_1X^{d-1}+\cdots +z_d$ is a non-zero polynomial of degree $d$ with complex coefficients and roots $\alpha_1,\ldots , \alpha_d$, the Mahler measure of $f$ is defined to be:
\begin{equation}\label{defMahler}
M(f)=|z_0| \prod_{h=1}^d \max\lg 1,|\alpha_h|\rg.
\end{equation}
Moreover, we set $M(0)=0$.

In what follows we are going to consider the Mahler measure as a function of the coefficients of a polynomial:
$$
\begin{array}{cccc}
M:& \R^{d+1} \text{ or }\C^{d+1} &\rightarrow & [0,\infty) \\
  & (z_0, \dots ,z_{d})& \mapsto & M\left( z_0X^d + \cdots +x_{d} \right).
\end{array}
$$

Mahler (\cite{Mahler1961}, Lemma 1) proved that such an $M$ is continuous and it is easy to see that it satisfies conditions \ref{cond1}. and \ref{cond2}. of Definition \ref{defdist}. We now prove that it is a semialgebraic function.

\begin{lemma}
The Mahler measure $M$, as a function of the coefficients of a polynomial, is a semialgebraic function. 
\end{lemma}

\begin{proof}
We start by proving the claim for the complex Mahler measure. We need to prove that, for every positive integer $n$, the function 
$$
\begin{array}{cccc}
M_n:& \R^{2(n+1)}&\rightarrow & [0,\infty) \\
  & (x_0, \dots ,x_{2n+1})& \mapsto & M\left( (x_0+ix_1)X^n + \cdots +(x_{2n}+ix_{2n+1}) \right)
\end{array}
$$
is semialgebraic, i.e., its graph
\begin{equation*}\label{graphMahler}
\Gamma(M_n) = \lg \left( x_0, \ldots ,x_{2n+1}, t\right) \in \R^{2(n+1)+1}: M\left(x_0, \ldots ,x_{2n+1}\right)=t \rg 
\end{equation*}
is a semialgebraic set. 

We prove this by induction on $n$. For $n=1$,
$$
\Gamma(M_1) = \lg \left( x_0, x_1,x_2,x_3, t\right) \in \R^{5}: \max \lg x_0^2+x_1^2 , x_2^2+x_3^2  \rg =t^2 ,t\geq 0 \rg 
$$
is clearly semialgebraic. Now suppose $n>1$. Let $\Gamma(M_n)=A \cup B$, where 
$$
A= \lg (x_0,\dots , x_{2n+1},t)\in \Gamma(M_n): x_0^2+x_1^2 \neq 0 \rg ,
$$
and
$$
B= \lg (x_0,\dots , x_{2n+1},t)\in \Gamma(M_n): x_0=x_1=0 \rg .
$$
By the inductive hypothesis, $B$ is a semialgebraic set since $B=\{ (0,0) \}\times \Gamma(M_{n-1})$.
Now let $A'$ be the set of points 
$$
(x_0,\dots , x_{2n+1},t,\alpha_1, \beta_1, \dots , \alpha_n , \beta_n)\in\R^{2(n+1)+1+2n},
$$
such that $ x_0^2+x_1^2 \neq 0$, $\alpha_h+i\beta_h$, for $h=1,\dots ,n$, are the roots of $(x_0+ix_1)X^n + \cdots +(x_{2n}+ix_{2n+1})$ and 
\begin{equation}\label{mah}
|x_0+i x_1|\prod_{h=1}^{n} \max \lg 1,|\alpha_h+ i \beta_h|\rg=t .
\end{equation}
This set $A'$ is defined by the symmetric functions that link the coefficients of a polynomial with its roots and by (\ref{mah}). It is therefore semialgebraic. Since $A$ is the projection of $A'$ on the first $2(n+1)+1$ coordinates, it is also semialgebraic by the Tarski-Seidenberg principle (Theorem \ref{tarski}). We have the claim for the complex Mahler measure.

For the real one it is sufficient to note that its graph is nothing but the projection that forgets the coordinates $x_1,x_3 ,\dots, x_{2n-1} ,x_{2n+1}$ of
$$
\Gamma(M_n)\cap \{ (x_0, \dots ,x_{2n+1},t): x_{2j+1}=0 \mbox{ for } j=0, \dots ,n \}.
$$
\end{proof}

Since $M$ satisfies the three conditions of Definition \ref{defdist}, it is a semialgebraic distance function. Moreover, in \cite{Chern2001}, Chern and Vaaler calculated the volume of the sets of the form (\ref{sets}) for the real and the complex monic Mahler measure. By (1.16) and (1.17) of \cite{Chern2001}, for every $T\geq 1$ the volumes of the sets
$$
\{(z_1,\ldots ,z_n) \in \R^n:M(1,z_1, \ldots ,z_n)\leq T\},
$$
and
$$
\{(z_1,\ldots ,z_n) \in \C^n:M(1,z_1, \ldots ,z_n)\leq T\}
$$
are, respectively, polynomials $p_\R(T)$ and $p_\C(T)$ of degree $n$ and $2n$ and leading coefficients
$$
C_{\R,n} = 2^{n-M} \left( \prod_{j=1}^{M}\left( \frac{2j}{2j+1} \right)^{n-2j} \right) \frac{n^M}{M!}, \footnote{There is a misprint in (1.16) of \cite{Chern2001}, $2^{-N}$ should read $2^{-M}$.}
$$
with $M=\lfloor \frac{n-1}{2} \rfloor$, and 
$$
C_{\C,n} = \pi^n \frac{n^n}{\left( n! \right)^2}.
$$

We just showed that $\mc{N}$ satisfies the hypothesis of Theorem \ref{mainthm} and we have that for every $\cH_0>1$ there exists a positive $C_0=C_0(m,n,\fN(S),\cH_0)$, such that for every $\cH\geq \cH_0$, 
\begin{equation}\label{asymptformMahler}
\left| \left|\Oseen_S^{\mc{N}}(\cH)\right|-   C_{\R,n}^rC^s_{\C,n} B_{\ks, S}^{(n)}  \mc{H}^{\m n} \left( \log  \mc{H} \right)^{|S|-1}  \right|  
  \leq \left\lbrace
\begin{array}{ll}
C_0   \mc{H}^{m n} \left( \log  \mc{H}  \right)^{|S|-2},  & \mbox{if $|S|> 1$,}\\
C_0 \mc{H}^{\m n-1}, & \mbox{if $|S|=1$,}
\end{array} \right.
\end{equation}
where $B^{(n)}_{\ks,S}$ is the constant defined in (\ref{const}).

Let us reformulate these considerations in terms of polynomials. We proceed in a similar way as done in Section 2 of \cite{algint}. For any positive integer $n$ we fix the system $\cN_n$ of dimension $n$ to consist of Mahler measure distance functions and we define 
$$
\begin{array}{cccc}
M^\ks:& \ks[X]& \rightarrow & [0,\infty) \\
 & a_0X^n+a_1X^{n-1}+\cdots +a_n  & \mapsto & H_{\cN_n}(a_0,a_1, \ldots , a_n).
\end{array}
$$
Therefore we can write
$$
M^k(a_0, \ldots , a_n)=\left( \prod_{i=1}^{r+s} M(\sigma_i(a_0)X^n+\cdots +\sigma_i(a_n))^{\frac{d_i}{\m}} \right)\prod_{v \nmid \infty} \max  \left\lbrace |a_0|_v, \ldots , |a_n|_v \right\rbrace^{\frac{d_v}{\m}}.
$$

Let $\mathcal{M}_{\ks, S}(\n , \cH)$ be the set of of monic polynomials $f \in \Oseen_S[X]$ of degree $n$ with $M^\ks(f)\leq \mc{H}$. Clearly $\left|\Oseen_S^\mc{N}(\cH)\right|=|\mathcal{M}_{\ks, S}(\n , \cH)|$ and  (\ref{asymptformMahler}) is an estimate for such cardinality. Fixing $m$, $n$, $|S|$ and an $|S|$-tuple of prime powers, and letting $\ks $ vary among all number fields of degree $m$, and $S$ among the sets of places of the chosen number field with the prescribed set of norms of the non-archimedean places, the constants $C_{\R,n}^r$, $C^s_{\C,n}$ and  $B_{\ks, S}^{(n)}$ are bounded and therefore there exists a constant $G_{m,\fN(S)}^{(n)}$, depending on $n$, $m$ and $\fN(S)$, such that
\begin{align}\label{upbound}
\left|\mathcal{M}_{\ks, S}(\n , \cH)\right| \leq G_{m,\fN(S)}^{(n)} \mc{H}^{\m n} \left( \log  \mc{H}+1 \right)^{|S|-1} ,
\end{align}
for every $\cH\geq 1$.

Note that, for every $\alpha \in \ks$, 
\begin{equation}\label{mahlheight}
M^\ks(X-\alpha)=\prod_{v \in M_\ks } \max \left\lbrace 1, |\alpha|_v\right\rbrace^{\frac{d_v}{\m}}=H(\alpha).
\end{equation}
%where $H$ is the usual absolute height function.

It is clear from the definition of Mahler measure (\ref{defMahler}) that
$$
M(fg)=M(f)M(g),
$$
and therefore, by Lemma 1.6.3 of \cite{BombGub}, one can see that
$$
M^\ks(fg)=M^\ks(f)M^\ks(g),
$$
for every $f,g \in \ks[X]$.

Now we want to restrict to monic $f$ irreducible over $\ks$. Let $\widetilde{\mathcal{M}}_{\ks, S}(\n , \cH)$ be the set of monic irreducible polynomials $f \in \Oseen_S[X]$ of degree $n$ with $M^\ks(f)\leq \mc{H}$, i.e., the polynomials in $\mathcal{M}_{\ks, S}(\n , \cH)$ that are irreducible over $\ks$.

\begin{corollary}
For every $\cH_0>1$ there exists a positive $D_0$, depending on $n$, $m$, $\fN(S)$ and $\mc{H}_0$, such that for every $\cH\geq \cH_0$ we have
$$
\left|\left|\widetilde{\mathcal{M}}_{\ks, S}(\n , \cH)\right|- C_{\R,n}^rC^s_{\C,n} B_{\ks, S}^{(n)}  \mc{H}^{\m n} \left( \log  \mc{H} \right)^{|S|-1} \right|  
\leq \left\lbrace
\begin{array}{ll}
D_0   \mc{H}^{m n} \left( \log  \mc{H}  \right)^{|S|-2},  & \mbox{if $|S|> 1$,}\\
 D_0 \mc{H}^{\m n-1}\mc{L}, & \mbox{if $|S|=1$,}
\end{array} \right.
$$
where $\mc{L}=\log \cH$ if $(m,n)=(1,2)$ and 1 otherwise.
\end{corollary}

\begin{proof}
For $n=1$, there is nothing to prove. Suppose $n>1$. We show that, up to a constant, the number of all monic reducible $f \in \Oseen_S[X]$ of degree $n$ with $M^\ks(f)\leq \mc{H}$ is not larger than the right hand side of (\ref{asymptformMahler}), except for the case $|S|=1$ and $(m,n)=(1,2)$. 

Consider all $f=gh \in \mathcal{M}_{\ks, S}(\n , \cH)$ with $g,h\in \Oseen_S[X]$ monic of degree $a$ and $b$ respectively, with $0<a\leq b<n$ and $a+b=n$.
We have $1 \leq M^\ks(g),M^\ks(h) \leq \cH$ because $g$ and $h$ are monic. Thus, there exists a positive integer $d$ such that $2^{d-1}\leq M^\ks(g)<2^d$. Note that $d$ must satisfy
\begin{equation}\label{ineqk}
1\leq d \leq \frac{\log \mc{H}}{\log 2}+1\leq 2 \log \mc{H} +1.
\end{equation}
Since $M^\ks$ is multiplicative,
$$
M^\ks(h)=\frac{M^\ks(f)}{M^\ks(g)}\leq 2^{1-d} \mc{H}.
$$
Using (\ref{upbound}) and noting that $2^d\leq 2 \cH$, we can say that there are at most
$$
G_{m,\fN(S)}^{(a)}\left(2^d\right)^{\m a } \left( \log  2^d+1 \right)^{|S|-1} \leq G_{m,\fN(S)}^{(a)}\left(2^d\right)^{\m a } \left( \log \cH +2 \right)^{|S|-1}  
$$
possibilities for $g$ and
$$
G_{m,\fN(S)}^{(b)} \left(2^{1-d} \mc{H}\right)^{\m b}\left( \log \left(2^{1-d} \mc{H} \right)+1 \right)^{|S|-1}   \leq G_{m,\fN(S)}^{(b)}\left(2^{1-d} \mc{H}\right)^{\m b}\left( \log  \mc{H} +2 \right)^{|S|-1}
$$
possibilities for $h$. Therefore, we have at most
\begin{align} \label{red}
H_{m,\fN(S)}^{(n)}\mc{H}^{\m b} 2^{\m d(a-b)} \left(  \log \mc{H} +2 \right)^{2(|S|-1)}
\end{align}
possibilities for $gh$ with $M^\ks (gh)\leq \cH$ and $2^{d-1}\leq M^\ks(g)<2^d$, where $H_{m,\fN(S)}^{(n)}$ is a real constant depending on $n$, $m$ and $\fN(S)$.

If $a=b=\frac{n}{2}$, then (\ref{red}) is
$$
H_{m,\fN(S)}^{(n)} \mc{H}^{\m \frac{n}{2}}  \left(  \log \mc{H} +2 \right)^{2(|S|-1)}.
$$
Summing over all $d$, $1\leq d \leq  \lfloor 2 \log \mc{H} \rfloor+1$ (recall (\ref{ineqk})), gives an extra factor $2\log \mc{H} +1 $. Therefore, when $a=b$, there are at most
$$
H_{m,\fN(S)}^{(n)} \mc{H}^{\frac{\m n}{2}}\left(2 \log \mc{H} +2 \right)^{2|S|-1}
$$
possibilities for $f=gh$, with $M^\ks(f)\leq \cH$. If $|S|>1$ or $(m,n)\neq (1,2)$, this has smaller order than the right hand side of (\ref{asymptformMahler}), since $\m n >2$ implies $\frac{\m n}{2}< \m n-1$. In the case $|S|=1$ and $(m,n)= (1,2)$, we get $H_{m,\fN(S)}^{(n)} \cH \left(2 \log \mc{H} +2 \right)$ and we need an additional logarithm factor.

In the case $a <b$, summing $2^{\m d(a-b)}$ over all $d$, $1\leq d \leq  \lfloor 2 \log \mc{H} \rfloor+1 =:D$, we get
$$
\sum_{d=1}^D \left(2^{\m(a-b)}\right)^d \leq \sum_{d=1}^{D}2^{-d}\leq 1.
$$
Thus, recalling $b\leq n-1$, if $a<b$ there are at most 
$$
H_{m,\fN(S)}^{(n)}\mc{H}^{\m(n-1)} \left( \log \mc{H} +2 \right)^{2(|S|-1)}
$$
possibilities for $f=gh$, with $M^\ks(f)\leq \cH$. This is again not larger than the right hand side of (\ref{asymptformMahler}).
\end{proof}

The last step of the proof links such irreducible polynomials with their roots and $M^\ks$ with the height of these roots.
Recall that $\overline{S}$ is the set of places of $\overline{\ks}$ that lie above the places in $S$.

\begin{lemma}
An algebraic number $\beta \in \Oseen_{\overline{S}}$ has degree $e$ over $\ks$ and $H( \beta ) \leq \cH$ if and only if it is a root of a monic irreducible polynomial $f \in \Oseen_S[X]$ of degree $e$ with $M^\ks(f)\leq \cH^e$. 
\end{lemma}

\begin{proof}
If an algebraic number $\beta \in \Oseen_{\overline{S}}$ has degree $e$ over $\ks$, then it is clearly a root of a monic irreducible polynomial $f \in \Oseen_S[X]$ of degree $e$, and vice-versa. We claim that
$$
H( \beta )^e=M^\ks(f).
$$

The function $M^\ks$ is independent of the choice of $k$ since it is possible to define an absolute $M^{\overline{\Q}}$ over $\overline{\Q}[X]$ that, restricted to any $k[X]$, coincides with $M^\ks$. To see this one can simply imitate the proof of the fact that the Weil height is independent of the field containing the coordinates (see \cite{BombGub}, Lemma 1.5.2).

%
%We show that it is possible to define an absolute $M^{\overline{\Q}}:\overline{\Q}[X]\rightarrow [0,\infty)$ such that, if $f \in \ks[X]$, then $M^{\overline{\Q}}(f)=M^\ks (f)$. In fact, let $\ks'$ be a finite extension of $\ks$ with $[\ks':\Q]=\m'$. Recall (see \cite{Neukirch1999}, Ch.II, (8.4) Corollary) that for any $w \in M_{\ks}$
%$$
%\sum_{\substack{v \in M_{\ks'}\\v \mid w}} d_v=d_w[\ks':\ks]=d_w\frac{\m'}{\m}.
%$$
%For any $f =a_0X^n+\dots +a_n \in \ks'[X]$,
%we use the notation $M_v(f)=M(f)$ for $v \mid \infty$ and $M_v(f)=\max\{ |a_0|_v , \dots , |a_n|_v \}$ for $v \nmid \infty$. We have
%\begin{align*}
%M^{\ks'}(f) & = \prod_{v \in M_{\ks'}} M_v(\sigma_v (f))^\frac{d_v}{\m'}=  \prod_{w \in M_{\ks}} \prod_{\substack{v \in M_{\ks'}\\v \mid w}}M_v(\sigma_v (f))^\frac{d_v}{\m'}\\
%&=\prod_{w \in M_{\ks}} M_w(\sigma_w (f))^{\sum_{\substack{v \in M_{\ks'}\\v \mid w}} \frac{d_v}{\m'}} 
%=\prod_{w \in M_{\ks}} M_w(\sigma_w (f))^\frac{d_w}{\m}=M^\ks(f).
%\end{align*}

Suppose $f=(X-\alpha_1)\cdots (X-\alpha_e)$. By (\ref{mahlheight}) we have
$$
M^{\Q(\alpha_i)}(X-\alpha_i)=H(\alpha_i),
$$
and the $\alpha_i$ have the same height because they are conjugate (see \cite{BombGub}, Proposition 1.5.17). Finally, by the multiplicativity of $M^\ks$ we can see that
$$
M^\ks(f)=M^{\overline{\Q}}(f)=\prod_{i=1}^e M^{\overline{\Q}}(X-\alpha_i)= H(\alpha_j)^e,
$$
for any $\alpha_j$ root of $f$.
\end{proof}

This implies that $|N(\Oseen_S(1,e),\cH)|=e \left|\widetilde{\mathcal{M}}_{\ks, S}(e, \cH^e)\right|$ because there are $e$ different $\beta \in  \Oseen_{\overline{S}}$ with the same minimal polynomial over $\ks$. We have that, for every $\cH_0>1$, there exists a positive $E_0=E_0(m,e,\fN(S),\mc{H}_0)$ such that, for every $\cH\geq \cH_0$,
\begin{multline*}
\left| N \left( \Oseen_S(1,e),\cH \right) -  e^{|S|}  C_{\R,e}^rC^s_{\C,e} B_{\ks, S}^{(e)} \mc{H}^{\m e^2} \left( \log  \mc{H} \right)^{|S|-1} \right| \\
\leq \left\lbrace
\begin{array}{ll}
E_0   \mc{H}^{m e^2} \left( \log  \mc{H}  \right)^{|S|-2},  & \mbox{if $|S|> 1$,}\\
 E_0 \mc{H}^{e(\m e-1)}\mc{L}, & \mbox{if $|S|=1$,}
\end{array} \right.
\end{multline*}
where $\mc{L}=\log \cH$ if $(m,e)=(1,2)$ and 1 otherwise.
We obtain Theorem \ref{thmint} by choosing $\cH_0=2$.

\section{Counting lattice points}

We start this section introducing the counting theorem that will be used to prove Theorem \ref{mainthm}. The principle dates back to Davenport \cite{Davenport1951} and was developed by several authors. In a previous work \cite{Barroero2012}, the author and Widmer formulated a counting theorem that relies on Davenport's Theorem and uses o-minimal structures. We do not need Theorem 1.3 of \cite{Barroero2012} in its full generality as we count lattice points in semialgebraic sets.

For a semialgebraic set $Z \subseteq \R^{n+n'}$, we call $Z_{\bo{t}}=\{ \x \in \R^n: (\x,\bo{t})\in Z\}$ the fiber of $Z$ lying above $\bo{t} \in \R^{n'}$ and $Z$ a semialgebraic family. It is clear that the fibers $Z_{\bo{t}}$ are semialgebraic subsets of $\R^n$. Let $\Lambda$ be a lattice of $\R^n$ with determinant $\det \Lambda$ and let $\lambda_i=\lambda_i(\Lambda)$, for $i=1,\ldots,n$, be the successive minima of $\Lambda$ with respect to the unit ball $B_0(1)$, i.e., 
\begin{alignat*}1
\lambda_i=\inf\{\lambda:B_0(\lambda)\cap\Lambda \text{ contains $i$ linearly independent vectors}\}.
\end{alignat*}

The following theorem is a special case of Theorem 1.3 of \cite{Barroero2012}.

\begin{theorem}\label{counttheorem}
Let $Z\subset \R^{n+n'}$ be a semialgebraic family and suppose the  fibers $Z_{\bo{t}}$ are bounded. Then there exists a constant $c_Z \in \R$, depending only on the family, such that
\begin{equation*}\label{eqcount}
\left| |Z_{\bo{t}}\cap \Lambda|-\frac{\Vol(Z_{\bo{t}})}{\det \Lambda} \right|\leq \sum_{j=0}^{n-1}c_{Z}\frac{V_j(Z_{\bo{t}})}{\lambda_1\cdots \lambda_j},
\end{equation*}
where $V_j(Z_{\bo{t}})$ is the sum of the $j$-dimensional volumes of the orthogonal projections of $Z_{\bo{t}}$ 
on every $j$-dimensional coordinate subspace of $\R^n$ and $V_0(Z_{\bo{t}})=1$.
\end{theorem}

Let us introduce the family we want to apply Theorem \ref{counttheorem} to. We fix an $(r,s)$-system $\cN$ of dimension $n$ consisting of $r$ real and $s$ complex semialgebraic distance functions. Recall that we defined $\tN_i(\z)=N_i(1,\z)$. Moreover, we see the complex $\tN_i$ as functions from $\R^{2n}$, i.e.,
$$
\tN_i(x_1,x_2,\dots, x_{2n-1},x_{2n})=\tN_i(z_1,\dots ,z_n),
$$ 
for $(x_1,x_2,\dots, x_{2n-1},x_{2n})=(\Re(z_1),\Im(z_1),\dots ,\Re(z_n),\Im(z_n))$.

Recall that $d_i=1$, for $i=1,\dots ,r$, and $d_i=2$, for $i=r+1,\dots ,r+s$, and $m=r+2s$. Let
\begin{equation}\label{defsetprod}
Z=\left\lbrace (\bo{x}_1, \ldots , \bo{x}_{r+s},t) \in \R^{n(r+2s)+1} : \prod_{i=1}^{r+s} \tN_i(\bo{x}_i)^{d_i}\leq t \right\rbrace ,
\end{equation}
where $\x_i \in \R^{d_i n}$.

We need to show that $Z$ is a semialgebraic family and that the fibers $Z_t$ are bounded for every $t\in \R$.

\begin{lemma}

The set $Z$ defined in (\ref{defsetprod}) is semialgebraic.

\end{lemma}

\begin{proof}

First note that, since the $N_i$ are semialgebraic functions, also the $\tN_i$ are semialgebraic. Indeed, one can get $\Gamma\left(\tN_i\right)$ by intersecting $\Gamma(N_i)$ with an appropriate affine subspace.
Let us define the following sets:
$$
S^{(i)}=\lg \left( \x_1,\ldots ,\x_{r+s},t,t_1,\ldots ,t_{r+s} \right) \in \R^{mn} \times \R^{1+r+s}: \tN_i(\x_i)= t_i  \rg,
$$
for $i=1,\ldots ,r+s$, and
$$
A=\lg \left( \x_1,\ldots ,\x_{r+s},t,t_1,\ldots ,t_{r+s} \right) \in \R^{mn} \times \R^{1+r+s}: \prod_{i=1}^{r+s} t_i^{d_i}\leq t   \rg.
$$
All these sets are clearly semialgebraic. Let $\pi$ be the projection map of $\R^{{mn}+1+r+s}$ to the first ${mn}+1$ coordinates. By the Tarski-Seidenberg principle (Theorem \ref{tarski}) the set
$$
B=\pi \left(\bigcap_i S^{(i)} \cap A \right)
$$
is semialgebraic. A point $\left( \x_1,\ldots ,\x_{r+s},t \right) $ belongs to $B$, if and only if there are $t_1, \ldots ,t_{r+s}$ such that $\tN_i(\x_i)= t_i$ for every $i$ and $\prod_{i=1}^{r+s} t_i^{d_i}\leq t$, i.e.,
 $\prod_{i=1}^{r+s} \tN_i(\bo{x}_i)^{d_i}\leq t$. Therefore $B=Z$, and we proved the claim.
\end{proof}

Since the $N_i$ are bounded distance functions, there exist positive real constants $\delta_i$ such that
$$
\delta_i |\z|_\infty \leq N_i(\z),
$$
for every $\z$ in $\R^{n+1} $ or $\C^{n+1}$ (see \cite{Cassels1971}, Lemma 2, p. 108). We define $\gamma_i=\max \{ \delta_i: \delta_i |\z|_\infty \leq N_i(\z) \}$ and $N_i'(\z)=\gamma_i |\z|_\infty$. As before, we use the notation $\tN_i'(\z)$ for $N_i'(1,\z)$. 

Let $\mc{N}'$ be the $(r,s)$-system consisting of $N'_i(\z)=\gamma_i|\z|_\infty$ for every $i=1,\dots ,r+s$. Each $(\bo{x}_1, \ldots , \bo{x}_{r+s},t)$ such that $\prod_{i=1}^{r+s} \tN_i(\bo{x}_i)^{d_i}\leq t $ satisfies $\prod_{i=1}^{r+s} \tN'_i(\bo{x}_i)^{d_i}\leq t $. Therefore, if 
$$
Z'=\left\lbrace (\bo{x}_1, \ldots , \bo{x}_{r+s},t) \in \R^{{mn}+1} : \prod_{i=1}^{r+s} \tN'_i(\bo{x}_i)^{d_i}\leq t \right\rbrace ,
$$ 
 we have $Z \subseteq Z'$. For every $\bo{x} \in \R^{d_i n}$ we have, by definition, $\tN'_i(\bo{x})\geq \gamma_i$ and therefore, for every $(\bo{x}_1, \ldots , \bo{x}_{r+s}) \in Z'_t$,
$$
\tN'_i(\bo{x}_i)^{d_i}\leq \frac{t}{\prod_{j \neq i} \gamma_j^{ d_j}}
$$
holds. This implies
$$
|\bo{x}_i|^{d_i}_\infty \leq \frac{t}{\prod_{j } \gamma_j^{ d_j}},
$$
for every $i=1,\dots , r+s$. We have just showed that the fibers $Z'_t$, and therefore $Z_t$, are bounded.

From now on we use the notation $Z(T)$ for $Z_T$. Recall that $V_j(Z(T))$ is the sum of the $j$-dimensional volumes of the orthogonal projections of $Z(T)$ 
on every $j$-dimensional coordinate subspace of $\R^n$ and $V_0(Z(T))=1$.

Since $Z \subseteq Z'$, we have $V_j(Z(T))\leq V_j(Z'(T))$.
By Theorem \ref{counttheorem} there exists a constant $c_Z$, depending only on $Z$, such that
\begin{equation}\label{ineq1}
\left| \left|Z(T) \cap \Lambda \right|- \frac{\Vol( Z(T) )}{\det \Lambda} \right| \leq \sum_{j=0}^{{mn}-1}c_{Z}\frac{V_j(Z'(T))}{\lambda_1\cdots \lambda_j},
\end{equation}
for every $T\in \R$. 

We have to calculate $\Vol( Z(T) )$ and we need upper bounds for $V_j(Z'(T))$. 

Recall we supposed that, for every $i=1,\dots, r+s$, $\tN_i(\x)\geq 1$ and the volume of the set $Z_i(T)$ defined in (\ref{sets}) is $p_i(T)$ for every $T\geq 1$, where $p_i$ is a polynomial of degree $d_i n$ and leading coefficient $C_i$.

\begin{lemma}\label{lemvol}

Let $q=r+s-1$. Under the hypotheses above we have that, for every $T \geq 1$,
\begin{equation*}\label{eqvol}
\Vol \left( Z(T)\right)=Q \left( T^\frac{1}{2},\log T \right),
\end{equation*}
where $Q(X,Y)\in \R[X,Y]$, $\deg_X Q = 2 \en$, $\deg_Y Q = q$ and the coefficient of $X^{2\en} Y^q$ is $\frac{\en^q }{q!} \prod_{i=1}^{q+1} C_i$.

\end{lemma}

\begin{proof}
This is a special case of Lemma 5.2 of \cite{algint}.
\end{proof}

The $V_j(Z'(T))$ were already computed in \cite{algint}.

\begin{lemma}\label{lemvolZ'}
For each $j=1, \dots , mn-1$, there exists a polynomial $P_j(X,Y)$ in $\R [X,Y]$, with $\deg_X P_j \leq 2 \en$, $\deg_Y P_j \leq q$, and the coefficient of $X^{2\en} Y^q$ is 0, such that, for every $T\geq 1$, we have
\begin{equation*}
V_j(Z'(T))= P_j\left(T^{\frac{1}{2}}, \log T \right).
\end{equation*}
\end{lemma}

\begin{proof}
See \cite{algint}, Lemma 5.4.
\end{proof}

For an integer $u$, we will use the notation
$$
X^{(u)}=\left\lbrace
\begin{array}{ll}
X^u, & \mbox{ for $u>0$,}\\
1, & \mbox{ for $u \leq 0$,}
\end{array}
\right.
$$
in order to avoid possible appearances of $0^0$, for instance in the following proposition, where we must consider $(\log T)^q$ for $T\geq 1$ and $q$ can be 0. 

Moreover, for $\Lambda$ a lattice, we define
$$
\mathfrak{D}(\Lambda)=\frac{1}{\det \Lambda}+\sum_{j=0}^{{mn}-1}\frac{1}{\lambda_1 \dots \lambda_j}
$$

\begin{proposition}\label{mainprop}
Let $\mc{N}$ be a $(r,s)$-system of dimension $n$ that satisfies the above hypotheses on the volumes of the sets $Z_i(T)$ and $\Lambda$ a lattice. There exist two positive real constants $E$ and $E'$, depending only on $\mc{N}$, such that, for every $T \geq 1$,
$$
\left| \left|Z(T) \cap \Lambda \right|-  \frac{n^q\prod_{i=1}^{q+1} C_i }{q!\det\Lambda}   T^n \left( \log T \right)^{(q)} \right|  
 \leq \left\lbrace
\begin{array}{ll}
 \mathfrak{D}(\Lambda)\left( E T^n \left( \log T \right)^{(q-1)}+E'\right), & \mbox{if $q\geq 1$,}\\
 \mathfrak{D}(\Lambda) E  T^{n-\frac{1}{m}},& \mbox{if $q=0$.}
\end{array} \right.
$$
Moreover, if $T<1$, then $Z(T)=\emptyset$. 
\end{proposition}

\begin{proof}
For $T<1$, $Z(T)=\emptyset$ since we supposed $\tN_i(\bo{x}) \geq 1$ for every $\x$. Suppose $T \geq 1$.

We start with the case $q=0$. In this case, our system $\mc{N}$ consists only of one function $N_1$ that can be either real ($d_1=m=1$) or complex ($d_1=m=2$). In any case, the volume of the set $Z(T) \subseteq \R^{m n}$ equals $p_1\left(T^{\frac{1}{m}}\right)$ for every $T \geq 1$, where $p_1$ has degree $m n $ and leading coefficient $C_1$.

Fix a $j$, $1\leq j \leq mn-1$. Any projection of $Z'(T)$ to a $j$-dimensional coordinate subspace has volume at most $F_jT^{\frac{j}{m}}$, for some positive real constant $F_j$. Therefore, there exists an $E''$ such that
$$
V_j(Z'(T))\leq E'' T^{n-\frac{1}{m}},
$$
for every $T \geq 1$, and by (\ref{ineq1}) we have the claim if $q=0$.

Suppose $q>0$. By (\ref{ineq1}), Lemma \ref{lemvol} and Lemma \ref{lemvolZ'}, we have the following inequality, for every $T\geq 1$,
\begin{equation*}
\left| \left|Z(T) \cap \Lambda \right|- \frac{n^q\prod_{i=1}^{q+1} C_i }{q!\det\Lambda}   T^n \left( \log T \right)^{(q)} \right| \leq 
 \mathfrak{D}(\Lambda) P\left(T^{\frac{1}{2}},\log T\right), 
\end{equation*}
for some polynomial $P(X,Y)\in \R[X,Y]$ with $\deg_XP\leq 2n$, $\deg_YP\leq q$, whose coefficients depend on $\mc{N}$ and the coefficient of $X^{2n}Y^q$ is 0. Since $P$ satisfies such conditions, there exists a positive $E$ such that 
$$
P\left(T^{\frac{1}{2}},\log T\right)\leq E T^n \left( \log T \right)^{(q-1)},
$$
for every $T\geq 3$. For $T \in [1,3]$, the function of $T$ given by $P\left(T^{\frac{1}{2}},\log T\right)$ is bounded, say by $E'$. Then 
$$
P\left(T^{\frac{1}{2}},\log T\right)\leq E T^n \left( \log T \right)^{(q-1)}+E',
$$
for every $T\geq 1$. Clearly, $E$ and $E'$ depend only on the coefficients of $P$ and therefore only on $\cN$. 
\end{proof}

\section{Proof of Theorem \ref{mainthm}}

In this section we prove Theorem \ref{mainthm}.

Recall that we fixed a number field $\ks$ of degree $\m$ over $\Q$. Let $\sigma_1,\dots ,\sigma_{r}$ be the real embeddings of $k$ and $\sigma_{r+1},\dots ,\sigma_{r+2s}$ be the complex ones, indexed in such a way that $\sigma_i=\overline{\sigma_{i+s}}$, for every $i=r+1,\dots, r+s$. For $\bo{a}=(a_1,\dots, a_n)\in k^n$, we set $\sigma_i(\bo{a})=(\sigma_i(a_1),\dots, \sigma_i(a_n))\in \R^n$ for $i=1,\dots, r$ and $\sigma_i(\bo{a})=(\Re( \sigma_i(a_1)),\Im( \sigma_i(a_1)),\dots,\Re( \sigma_i(a_n)),\Im( \sigma_i(a_n)))\in \R^{2n}$ for $i=r+1,\dots, r+s$.

Let $\mathfrak{A}$ be a non-zero fractional ideal of $\ks$. The image of $\A$ via the embedding $\sigma: a \hookrightarrow (\sigma_1 (a), \ldots ,\sigma_{r+s} (a))$ is a lattice in $\R^\m$. If we set  $\Lambda_\A=\tau(\A^n)$, where $\tau(\bo{a})=(\sigma_1 (\bo{a}), \ldots ,\sigma_{r+s} (\bo{a}))$, for $\bo{a}\in k^n$, then $\Lambda_\A$ is a lattice in $\R^{mn}$. Recall that $\fN(\A)$ indicates the norm of $\A$ and $\Delta_k$ the discriminant of $k$.

\begin{lemma}\label{lemdet}
We have 
$$
\det \Lambda_\A =\left( 2^{-s} \fN(\A) \sqrt{|\Delta_\ks|}\right)^{\n},
$$
and the first successive minimum of $\Lambda_\A $ with respect to the Euclidean distance is $\lambda_1 \geq \fN(\A)^{\frac{1}{\m}}$.
\end{lemma}

\begin{proof}
In \cite{Masser2007} this Lemma is stated for integral ideals (\cite{Masser2007}, Lemma 5). The same arguments work also for non-zero fractional ideals. 
\end{proof}

To prove Theorem \ref{mainthm} we need an estimate for the cardinality of $\Oseen_S^{\mc{N}}(\cH)$, i.e., the set of points $\bo{a} \in \Oseen_S^n$ such that $H_\mathcal{N} (1,\bo{a}) \leq \mc{H}$.

Recall that we set $d_i=1$, for $i=1,\dots ,r$, and $d_i=2$, for $i=r+1,\dots ,r+s$.
As in Section \ref{intro}, we call $S_\fin$ the set of non-archimedean places in $S$. 

First suppose $S_\fin=\emptyset$, then $\Oseen_S=\Oseen_\ks$ and $|S|=q+1=r+s$. Note that, if $\bo{a}$ is a vector with integer coordinates, its non-archimedean absolute values are smaller than or equal to 1. Then
$$
H_{\cN}(1,\bo{a})
=\prod_{v \in M_k}N_v(1,\sigma_v(\bo{a}))^{\frac{d_v}{\m}}= \prod_{i=1}^{r+s}\tN_i(\sigma_i(\bo{a}))^{\frac{d_i}{\m}} ,
$$
for every $\bo{a} \in \Oseen_k^n$.
Therefore, the number of $\bo{a} \in \Oseen_\ks^n$ such that $H_{\cN}(1,\bo{a} )\leq \mc{H}$ is the number of lattice points of $\Lambda_{\Oseen_\ks}=\tau(\Oseen_\ks^n)$ in $Z(\mc{H}^\m)$. By Lemma \ref{lemdet}, $\det \Lambda_{\Oseen_\ks}= \left(  2^{-s} \sqrt{|\Delta_{\ks}|} \right)^\n$ and $\lambda_1\geq 1$. Thus, $\mathfrak{D}(\Lambda_{\Oseen_\ks})\leq {mn}+2^{sn}$. Moreover, for every $\cH_0>1$ there exists a $C_0=C_0(\cN,\cH_0)$ such that, if $q\geq 1$, 
$$
({mn}+2^{sn})\left( E \cH^{\m\n} \left( \log \cH^{\m} \right)^{(q-1)}+E'\right)\leq C_0 \cH^{\m\n} \left( \log \cH \right)^{(q-1)},
$$
for every $\cH\geq \cH_0$ and, in case $q=0$, $({mn}+2^{sn}) E\leq C_0 $. The claim of Theorem \ref{mainthm} follows applying Proposition \ref{mainprop}.\\

From now, to avoid confusion between Cartesian powers and powers of an ideal with respect to the operation of ideal multiplication, we indicate the latter by $\A^{\star(d)}$ for a non-zero fractional ideal $\A$ and an integer $d$.

Now, suppose $S_\fin=\{ v_1, \dots ,v_L\}$, with $L>0$. In this case we cannot apply Proposition \ref{mainprop} to $\tau(\Oseen_S^n)$ directly because it is dense in $\R^{mn}$.

Recall that $v_l$ corresponds to the prime ideal $\fp_l$ of $\Oseen_\ks$. Let $\mc{I}_{S}$ be the set of non-zero integral ideals $\A$ in $\Oseen_\ks$ which are products of the prime ideals we fixed, i.e., $\A=\fp_1^{\star(g_1)} \dots \fp_L^{\star(g_L)}$ for some non-negative integers $g_1, \dots ,g_L$. An $\bo{a} \in \ks^n$ is in $\Oseen_S^n$ if and only if there exists an ideal $\A \in \mc{I}_{S}$ such that $a_u \in \A^{\star(-1)}$ for every $u=1,\dots ,n$, i.e., $\tau(\bo{a})=(\sigma_1(\bo{a}), \dots ,\sigma_{r+s}(\bo{a})) \in \Lambda_{\A^{\star(-1)}}$ which is a lattice in $\R^{mn}$.  We will therefore apply Proposition \ref{mainprop} to lattices of this form and then combine the obtained estimates.

We set
$$
V_{\ks,\cN} = \frac{n^q 2^{sn} }{q!\left( \sqrt{|\Delta_\ks|}\right)^{n}} \prod_{i=1}^{q+1} C_i .
$$
For a non-zero integral ideal $\A$ and $T>0$, by $Z(\A,T)$ we indicate the set of $\bo{a} \in \ks^n$ such that $\tau(\bo{a}) \in \Lambda_{\A^{\star(-1)}} \cap Z(T^\m)$.

\begin{lemma}\label{lemZA}
There exist two positive constants $F$ and $F'$, depending only on $\cN$ such that, for $T\geq 1$ and every non-zero integral ideal $\A$, we have
\begin{multline*}
\left| \vphantom{\left( \log T^\m \right)^{(q)} }|Z(\A,T)|- V_{\ks,\cN} \fN(\A)^nT^{\m n} \left( \log T^\m \right)^{(q)}  \right|  \\
\leq \left\lbrace
\begin{array}{ll}
 \fN(\A)^n \left( F T^{\m n} \left( \log T^\m \right)^{(q-1)}+F'\right), & \mbox{if $q\geq 1$,}\\
 \fN(\A)^n F T^{\m n-1},& \mbox{if $q=0$.}
\end{array}
\right.
\end{multline*}
Moreover, if $T<1$, $Z(\A,T)=\emptyset$.
\end{lemma}

\begin{proof}
Note that, by Lemma \ref{lemdet}, the first successive minimum of $\Lambda_{\A^{\star(-1)}}$ is greater than or equal to $\fN(\A)^{-\frac{1}{\m}}$. Since $\fN(\A)$ is a positive integer, we have
$$
\prod_{i=1}^{j}\lambda_i \geq \fN(\A)^{-\frac{j}{\m}} \geq \fN(\A)^{-\frac{{mn}-1}{\m}}=\fN(\A)^{-n+\frac{1}{\m}} \geq \fN(\A)^{-n},
$$
for every $j=1, \dots ,{mn}-1$. Moreover, $|\Delta_\ks|\geq 1$. The claim follows from Proposition \ref{mainprop} and Lemma \ref{lemdet}, after noting that 
$$
\mathfrak{D}\left(\Lambda_{\A^{\star(-1)}}\right)\leq {mn} \fN(\A)^n +\frac{ 2^{sn}\fN(\A)^n }{\left( \sqrt{|\Delta_\ks|}\right)^{n}}\leq \fN(\A)^n \left({mn}+2^{sn} \right).
$$
\end{proof}

We fix a $T\geq 1$. For a non-zero integral ideal $\A$, let $Z^*(\A,T)$ be the subset of $Z(\A,T)$ consisting of the points $\bo{a}$ such that, for every $\B $ strictly dividing $\A$, there is a $u\in \{1, \dots ,n\}$ such that  $a_u \not\in \B^{\star(-1)}$. In other words, $\bo{a}$ corresponds to a lattice point of $\Lambda_{\A^{\star(-1)}}$ that is not contained in any sublattice of the form $\Lambda_{\B^{\star(-1)}}$ where $\B$ is a strict divisor of $\A$. We have
$$
|Z(\A,T) |= \sum_{\B \mid \A}|Z^*(\B,T)|.
$$
If $\mu_\ks$ is the M\"obius function for the non-zero ideals of $\Oseen_k$, the M\"obius inversion formula implies that
$$
\left|Z^*(\A,T)\right| = \sum_{\B \mid \A} \mu_\ks(\B) \left|Z\left(\A\B^{\star(-1)},T\right)\right|.
$$
Lemma \ref{lemZA} gives us an estimate for $|Z^*(\A,T) |$, for every $T\geq 1$,
\begin{multline}\label{Zstarest}
\left| |Z^*(\A,T)|- V_{\ks,\cN}  \sum_{\B \mid \A} \mu_\ks(\B) \fN\left( \A\B^{\star(-1)}\right)^{n} T^{\m n} \left( \log T^\m \right)^{(q)}  \right|  \\
\leq \left\lbrace\begin{array}{ll}
 \sum_{\B \mid \A}| \mu_\ks(\B)| \fN\left(\A\B^{\star(-1)}\right)^{n} \left( F T^{\m n} \left( \log T^\m \right)^{(q-1)}+F'\right) , & \mbox{if $q\geq 1$,} \\
F \sum_{\B \mid \A} |\mu_\ks(\B)| \fN\left(\A\B^{\star(-1)}\right)^{n} T^{\m n-1}  ,& \mbox{if $q=0$,}
\end{array} \right.
\end{multline}
and $Z^*(\A,T)=\emptyset$ if $T<1$.

Recall that $\Oseen^\mc{N}_S(\cH)$ is the set of points $\bo{a} \in \Oseen_S^n$ with $H_\cN(1,\bo{a})\leq \cH$.

\begin{lemma}
For every $\mc{H} \geq 1$ we have
\begin{equation}\label{defosh}
\left| \Oseen^\mc{N}_S(\cH) \right|= \sum_{\substack{ \A \in \mc{I}_{S}, \\ \fN(\A)^{-1} \cH^\m\geq 1}}\left| Z^*\left(\A,\fN(\A)^{-\frac{1}{\m}} \cH\right)\right|.
\end{equation}
\end{lemma}

\begin{proof}
Let $\A=\fp_1^{\star(g_1)} \dots \fp_L^{\star(g_L)}$ and recall $d_{v_l}=[\ks_{v_l}:\Q_{v_l}]$ is the local degree of $k$ at $v_l$. Every point $\bo{a} \in Z^*(\A,T)$ is such that $\max_{u \in \{1,\dots ,n \}}| a_u|_{v_l}^{d_{v_l}}=\fN\left(\fp_l\right)^{g_l}$, for every $l=1,\dots , L$, and $\max_{u \in \{1,\dots ,n \}} | a_u|_{v}\leq 1$ for all $v \not\in S$. This means that every $\bo{a} \in Z^*(\A,T)$ satisfies
$$
\prod_{v \nmid \infty} \max_u\{1,| a_u|_v \}^{d_v}= \fN(\A),
$$
and thus
$$
H_\mc{N} (1,\bo{a})=\fN(\A)^{\frac{1}{\m}}\prod_{i=1}^{r+s}\tN_i(\sigma_i(\bo{a}))^{\frac{d_i}{\m}} \leq \fN(\A)^{\frac{1}{\m}} T.
$$
Therefore, $\bo{a} \in \Oseen^\mc{N}_S(\cH)$ if and only if there exists an $\A  \in \mc{I}_{S}$ such that $\bo{a} \in Z^*\left(\A,\fN(\A)^{-\frac{1}{\m}} \cH\right)$. Since such an $\A$ is unique and recalling that, if $T<1$, then $Z^*(\A,T)$ is empty, we obtain the claim.
\end{proof}

Let $\mc{I}_{S}(T)$ be the set of ideals in $\mc{I}_{S}$ with norm not exceeding $T$ and recall that the norm is multiplicative.
Combining (\ref{defosh}) with (\ref{Zstarest}), we have that
$$
\left| \left|\Oseen^\mc{N}_S(\cH)\right|- V_{\ks,\cN} \sum_{\A \in \mc{I}_{S}\left( \cH^\m\right)} \sum_{\B \mid \A} \frac{\mu_\ks(\B)}{\fN(\B)^{n} }   \cH^{\m n} \left( \log \left(\frac{\cH^\m}{\fN(\A)}\right) \right)^{(q)}  \right|  
$$
is smaller than or equal to
$$
 \sum_{\A \in \mc{I}_{S}\left( \cH^\m\right)}\sum_{\B \mid \A} \frac{|\mu_\ks(\B)|}{\fN(\B)^{n} }  \left(F \cH^{\m n} \left( \log \left( \frac{\cH^\m}{\fN(\A)}\right) \right)^{(q-1)} +F' \fN(\A)^{n} \right) 
 $$
 if $q\geq 1$ and
 $$
F \sum_{\A \in \mc{I}_{S}\left( \cH^\m\right)}\sum_{\B \mid \A} \frac{|\mu_\ks(\B)|}{\fN(\B)^{n} } \fN(\A)^{\frac{1}{\m}}  \cH^{\m n-1}
$$
if $q=0$, for every $\cH \geq 1$.

Now, let $\Psi^{(1)}(\A)=\sum_{\B \mid \A} \frac{\mu_\ks(\B)}{\fN(\B)^n}$ and $\Psi^{(2)}(\A)=\sum_{\B \mid \A} \frac{|\mu_\ks(\B)|}{\fN(\B)^n}$. Therefore
\begin{multline} \label{mainterm}
\left| \left|\Oseen^\mc{N}_S(\cH)\right|-V_{\ks,\cN} \cH^{\m n} \sum_{\A \in \mc{I}_{S}( \cH^\m) } \Psi^{(1)}(\A) \left( \log \left(\frac{\cH^\m}{\fN(\A)}\right) \right)^{(q)} \right| \\
\leq \left\lbrace    \begin{array}{ll}  
 \sum_{\A \in \mc{I}_{S}(  \cH^\m) } \Psi^{(2)}(\A)\left(F \cH^{\m n} \left( \log\left(\frac{\cH^\m}{\fN(\A)}\right) \right)^{(q-1)} +F' \fN(\A)^{n} \right) , & \mbox{if $q\geq 1$,} \\
F \cH^{\m n-1}  \sum_{\A \in \mc{I}_{S}(  \cH^\m) } \Psi^{(2)}(\A)\fN(\A)^{\frac{1}{\m}}  ,& \mbox{if $q=0$.}
\end{array} \right. 
\end{multline} 

Let $ K $ be a non-negative integer, we set
$$
\cL^{(h)}_{S}(\cH,K)=\sum_{\A \in \mc{I}_{S}(  \cH^\m) } \Psi^{(h)}(\A) \left( \log \left(\frac{\cH^\m}{\fN(\A)}\right) \right)^{(K)},
$$
for $h=1,2$. Recall that we defined $\fN(S)=(\fN(\fp_1),\dots , \fN(\fp_L))$, and let
$$
F^{(h)}_l=\frac{\Psi^{(h)}(\fp_l)}{\log \fN\left(\fp_l \right)} .
$$
In the next lemma we allow $S_{\fin}$ to be empty as the base step for the induction.

\begin{lemma}\label{lemlog}
For every non-negative integer $K$, there exists a positive constant $U_{K,\fN(S)}$, depending only on $K$ and $\fN(S)$, such that for $h=1,2$ and for every $\cH\geq 1$
$$
\left| \cL^{(h)}_{S}(\cH,K)- \left( \prod_{l=1}^L F^{(h)}_l \right)   \left( \prod_{i=K+1}^{K+L} \frac{1}{i}  \right)\vphantom{ \left( \prod_{l=1}^L F^{(h)}_l \right)   \left( \prod_{i=K+1}^{K+L} \frac{1}{i}  \right)} (\log \cH^\m )^{(K+L)} \right| 
 \leq U_{K,\fN(S)} \left( \log \cH^\m +1 \right)^{(K+L-1)} .
$$
\end{lemma}

%\begin{lemma}\label{lemlog}
%For every non negative integer $ K$, we have, for $\cH\rightarrow \infty$,
%$$
%\cL_{S'}(\cH,K)=\left( \prod_{l=1}^L F_l \right)   \left( \prod_{i=K+1}^{K+L} \frac{1}{i}  \right) (\log \cH^\m )^{K+L} +O((\log \cH^\m )^{K+L-1} ),
%$$
%where
%$$
%F_l=\frac{1}{f_l \log p_l} \left( 1-\frac{1}{\left(p_l^{f_l}\right)^n} \right) .
%$$
%If $L=0$ the products are intended to be 1.
%\end{lemma}

\begin{proof}

We proceed by induction on the cardinality of $S_{\fin}$. Clearly, we can define $\cL^{(h)}_{S'}(\cH,K)$ and $\mc{I}_{S'}$ for $S'= S\setminus \{v_L \}$. 

If $S_{\fin}$ is empty, i.e., $L=0$, then $ \mc{I}_{S}(  \cH^\m)=\{\Oseen_\ks\}$ and $\cL^{(h)}_{S}(\cH,K)=(\log \cH^\m )^{(K)}$, for every $\cH\geq 1$. 

Now suppose $S_{\fin}$ has cardinality $L>0$. The sum over all $\A \in \mc{I}_{S}(  \cH^\m)$ can be viewed as two sums: the first over all $\B \in \mc{I}_{S'}(  \cH^\m)$, and the second over all non-negative integers $g_L$, with $\fN\left(\fp_L^{\star(g_L)}\right)\leq  \cH^\m \fN(\B)^{-1}$.
For typographical convenience we set 
$$
A(\B)=\left\lfloor \frac{\log \left(  \cH^\m \fN(\B)^{-1}\right)}{\log \fN\left(\fp_L \right)  }\right\rfloor,
$$
and
$$
R= \mc{I}_{S'}(  \cH^\m).
$$
We have 
\begin{gather*}
 \cL^{(h)}_{S}(\cH,K)=\sum_{\B \in R} \sum_{g_L=0}^{A(\B)}\Psi^{(h)}\left(\B\fp_L^{\star(g_L)}\right)\left( \log \left(\frac{ \cH^\m}{\fN(\B)}\right)-g_L\log \fN\left(\fp_L \right) \right)^{(K)}\\ 
= \sum_{\B \in R} \sum_{g_L=1}^{A(\B)}\Psi^{(h)}\left(\B\fp_L^{\star(g_L)}\right) \sum_{i=0}^{K} (-1)^i \binom{K}{i} (\log \fN\left(\fp_L \right))^i g_L^i  \left( \log \left( \frac{ \cH^\m}{\fN(\B)}\right)\right)^{(K-i)} + \cL^{(h)}_{S'}(\cH,K).
\end{gather*} 
Using the definitions of $ \Psi^{(h)}$, it is easy to see that $1/2 \leq \Psi^{(h)}(\fp_l) \leq 3/2$ for every $l$ and, if $g_L\geq 1$,
\begin{equation}\label{Psi}
\Psi^{(h)}\left(\B\fp_L^{\star(g_L)}\right) = \Psi^{(h)}(\B\fp_L)  = \Psi^{(h)}(\B) \Psi^{(h)} (\fp_L)> 0.
\end{equation}
Therefore, 
\begin{gather} 
 \cL^{(h)}_{S}(\cH,K)  =  \Psi^{(h)} (\fp_L) \sum_{i=0}^{K} (-1)^i \binom{K}{i}    \left(\log \fN\left(\fp_L \right)\right)^i   \sum_{\B\in R} \Psi^{(h)}(\B)\left( \log \left( \frac{ \cH^\m}{\fN(\B)}\right)\right)^{(K-i)}\sum_{g_L=1}^{A(\B)}g_L^i \nonumber \\   \label{LS} + \cL_{S'}^{(h)}(\cH,K)   . 
\end{gather}
Using Faulhaber's formula, for every $i=0, \ldots ,K$, we have
$$
\sum_{g_L=1}^{A(\B)}g_L^i - \frac{1}{i+1}\left\lfloor \frac{\log \left( \cH^\m \fN(\B)^{-1}\right)}{\log \fN\left(\fp_L \right)}\right\rfloor^{i+1} = Q_i\left(\left\lfloor \frac{\log \left( \cH^\m \fN(\B)^{-1}\right)}{\log \fN\left(\fp_L \right)}\right\rfloor\right),
$$
where $Q_i $ is a polynomial of degree $i$ (except $Q_0 =0$) whose coefficients depend only on $i$. Then
$$
\left|\sum_{g_L=1}^{A(\B)}g_L^i - \frac{1}{i+1}\left( \frac{\log \left( \cH^\m \fN(\B)^{-1}\right)}{\log \fN\left(\fp_L \right)}\right)^{i+1}  \right| \leq Q_i'\left( \log \left(\frac{ \cH^\m}{\fN(\B)}\right)\right),
$$
where $Q_i' $ is a polynomial of degree at most $i$, whose coefficients depend on $i$ and $ \fN\left(\fp_L \right)$. Finally, after noting that
$$
 \sum_{i=0}^{K} (-1)^i \binom{K}{i}\frac{1}{i+1}=\frac{1}{K+1},
$$
by (\ref{LS}), we can derive the following inequality:
\begin{multline*}
\left|\cL^{(h)}_{S}(\cH,K) -\frac{F^{(h)}_L}{K+1}\sum_{\B\in R} \Psi^{(h)}(\B)\left( \log \left( \frac{ \cH^\m}{\fN(\B)}\right)\right)^{(K+1)} \right|  \\
\leq \cL_{S'}^{(h)}(\cH,K)+  \sum_{\B\in R} \Psi^{(h)}(\B) Q\left( \log \left( \frac{ \cH^\m}{\fN(\B)}\right)\right),
\end{multline*}
where $Q$ is a polynomial of degree at most $K$ whose coefficient depend only on $K$ and $ \fN\left(\fp_L \right)$.
Therefore, we have
$$
\left|\cL_{S}^{(h)}(\cH,K)-\frac{F^{(h)}_L}{K+1} \cL_{S'}^{(h)}(\cH,K+1)\right|\leq  \sum_{i=0}^{K}  b_i \cL_{S'}^{(h)}(\cH,i),
$$
where the $b_i$ are real coefficients again depending on $K$ and $ \fN\left(\fp_L \right)$. Now, by the inductive hypothesis, there exist $U_{K+1,\fN(S')}$ and $U'_{i,\fN(S')}$, for $i=0, \dots ,K$, such that
\begin{multline*}
\left| \cL_{S'}^{(h)}(\cH,K+1)-\left( \prod_{l=1}^{L-1} F^{(h)}_l \right)   \left( \prod_{i=K+2}^{K+L} \frac{1}{i}  \right) \left(\log \cH^\m \right)^{(K+L)}  \right|   \\
\leq U_{K+1,\fN(S')} \left(\log \cH^\m  +1 \right)^{(K+L-1)} ,
\end{multline*}
and
$$
\cL_{S'}^{(h)}(\cH,i) \leq U'_{i,\fN(S')}  \left(\log \cH^\m +1 \right)^{(i+L-1)} ,
$$
for every $i=0, \dots ,K$.
The claim follows easily.
\end{proof}

\begin{lemma} \label{lemq0}
There exists a real constant $V_{m,\fN(S)}$, depending only on $m$ and $\fN(S)$, such that 
$$
 \sum_{\A \in \mc{I}_{S}(  \cH^\m) } \Psi^{(2)}(\A)\fN(\A)^{\frac{1}{\m}} \leq V_{m,\fN(S)} \cH \left( \log \cH +1 \right)^{(L-1)} ,
$$
for every $\cH\geq 1$.
\end{lemma}

\begin{proof}
We proceed by induction on the cardinality of $S_{\fin}$ as before. If $S_\fin$ is empty, then $ \sum_{\A \in \mc{I}_{S}(  \cH^\m) } \Psi^{(2)}(\A)\fN(\A)^{\frac{1}{\m}}=1$ and the claim holds. Now suppose $S_\fin=\{ v_1, \dots , v_L \}$, with $L>0$, and again $\fp_1 , \dots , \fp_L$ are the prime associated to the places in $S_\fin$. Let $S'=S\setminus \{ v_L\}$ and again
$$
A(\B)=\left\lfloor \frac{\log \left(  \cH^\m \fN(\B)^{-1}\right)}{\log \fN(\fp_L)}\right\rfloor.
$$
Note that $\Psi^{(2)}(\fp_L)\leq 2$ and then, by (\ref{Psi}), $\Psi^{(2)}(\B\fp_L^{\star(g_L)})\leq 2 \Psi^{(2)}(\B)$. Then
\begin{align*}
\sum_{\A \in \mc{I}_{S}(  \cH^\m) } \Psi^{(2)}(\A)\fN(\A)^{\frac{1}{\m}} &\leq \sum_{\B \in \mc{I}_{S'}\left( \cH^\m\right) }2 \Psi^{(2)}(\B) \fN(\B)^{\frac{1}{\m}} \sum_{g_L=0}^{A(\B)} \fN(\fp_L)^{\frac{g_L }{\m}}\\
&=2\sum_{\B \in \mc{I}_{S'}\left( \cH^\m\right) }\Psi^{(2)}(\B) \fN(\B)^{\frac{1}{\m}} \frac{ \fN(\fp_L)^{\frac{1}{\m}(A(\B)+1)}-1}{ \fN(\fp_L)^{\frac{1 }{\m}}-1} \\
&\leq  \frac{2\fN(\fp_L)^{\frac{1 }{\m}}}{\fN(\fp_L)^{\frac{1 }{\m}}-1} \sum_{\B \in \mc{I}_{S'}\left( \cH^\m\right) } \Psi^{(2)}(\B)  \fN(\B)^{\frac{1}{\m}} \left( \fN(\fp_L)^{\frac{\log \left( \cH^\m \fN(\B)^{-1}  \right)}{\log \fN(\fp_L)}}  \right)^{\frac{1}{\m}}\\
&=\frac{2\fN(\fp_L)^{\frac{1 }{\m}}}{\fN(\fp_L)^{\frac{1 }{\m}}-1}\sum_{\B \in \mc{I}_{S'}\left( \cH^\m\right) } \Psi^{(2)}(\B)  \fN(\B)^{\frac{1}{\m}} \left(\frac{ \cH^\m }{\fN(\B)}  \right)^{\frac{1}{\m}}\\ 
& \leq \frac{2\fN(\fp_L)^{\frac{1 }{\m}}}{\fN(\fp_L)^{\frac{1 }{\m}}-1} \cH \cL_{S'}^{(2)}(\cH,0) .
\end{align*}
The claim follows applying Lemma \ref{lemlog}.
\end{proof}

Now we are ready prove Theorem \ref{mainthm}. 

We already dealt with the case $S_{\fin}=\emptyset$. Suppose $S_\fin \neq \emptyset$. By (\ref{mainterm}) we have
\begin{align*}
\left|\left|\Oseen^\mc{N}_S(\cH)\right|- V_{\ks,\cN}\cH^{\m n} \cL_{S}^{(1)}\left(\cH,q\right) \right|  
\leq \left\lbrace\begin{array}{ll}
F \cH^{\m n}\cL_{S}^{(2)}\left(\cH,q-1\right) + F'\cH^{\m n} \cL_{S}^{(2)}(\cH,0) , & \mbox{if $q\geq 1$,} \\
F \cH^{\m n-1}  \sum_{\A \in \mc{I}_{S}(  \cH^\m) } \Psi^{(2)}(\A)\fN(\A)^{\frac{1}{\m}}  ,& \mbox{if $q=0$.}
\end{array} \right.
\end{align*}

Note that, $L\leq |S|-1$ and if $q\geq 1$, then $L\leq |S|-2$. Moreover,
$$
F^{(1)}_l=\frac{\Psi^{(1)}(\fp_l)}{\log \fN\left(\fp_l \right)}=\frac{1}{\log \fN\left(\fp_l \right)}\left( 1- \frac{1}{\fN\left(\fp_l \right)^n}\right) .
$$
 We apply Lemmas \ref{lemlog} and \ref{lemq0} and we can conclude that there exists a positive $G=G(\cN,\fN(S))$ such that 
$$
\left|\left|\Oseen^n_S(\cH)\right|- C_{\cN,\ks,S}\mc{H}^{\m n} \left( \log  \mc{H} \right)^{|S|-1}  \right| \leq 
G \cH^{\m n} \left( \log \cH +1 \right)^{|S|-2},
$$
for every $\cH\geq 1$, where $C_{\cN,\ks,S}$ was defined in (\ref{constmaint}).

Now, for every $\cH_0>1$, there exists a positive $C_0$, clearly depending on $\cN$, $\fN(S)$ and $\cH_0$ such that 
$$
G \cH^{\m n} \left( \log \cH +1 \right)^{|S|-2}\leq C_0 \cH^{\m n} \left( \log \cH  \right)^{|S|-2},
$$
and we have the claim of Theorem \ref{mainthm}

\section*{Acknowledgements}

The author would like to thank Jeffrey Vaaler for many useful discussions and the hospitality at the Department of Mathematics at UT Austin and Martin Widmer for his encouragement and his advice that significantly improved this article.

\bibliographystyle{amsplain}
\bibliography{bibliography.bib}

\end{document}